\documentclass[12pt]{article}

\usepackage{amsmath,amsthm,amsfonts,amssymb}
\usepackage{graphicx}

\allowdisplaybreaks[4]

\newtheorem {lemma}{Lemma}[section]
\newtheorem {theorem} {Theorem}[section]

\newtheorem {corollary}{Corollary}[section]

\newtheorem {question}{Question}[section]
\usepackage[figurename=Fig.]{caption}

\usepackage{fullpage}

\begin{document}

\title{Answers to Gould's question concerning the existence of chorded cycles}

\author{Leyou Xu\footnote{Email: leyouxu@m.scnu.edu.cn}, Bo Zhou\footnote{Corresponding author;
Email: zhoubo@m.scnu.edu.cn}\\
School of Mathematical Sciences, South China Normal University\\
Guangzhou 510631, P.R. China}

\date{}
\maketitle

\begin{abstract}
Answers are offered to Gould's question to find spectral sufficient conditions
that imply a graph contains a chorded cycle via signless Laplacian spectral radius. The conditions are tight. 
\end{abstract}

\section{Introduction}

All graphs considered in this paper are simple and finite.
Let $C$ be a cycle in a graph. A chord of  $C$  is an edge not in $C$  between two vertices of $C$.
If $C$ has at least one chord, then it is called a chorded cycle. In 1961
 P\'{o}sa \cite{Pos} asked a question: What conditions imply a graph contains a chorded cycle?
P\'{o}sa offered an answer by proving that any graph of order $n\ge 4$ with at least $2n-3$ edges contains a chorded cycle. In 1963
Czipser (see problem 10.2, p.~65 in \cite{Lov}) gave another simple answer by proving that  any graph with minimum degree at least
three contains a chorded cycle. In recent years,
the study of cycles with certain properties such as chorded cycles in graphs received much attention.
There are various sufficient conditions for the existence of a
chorded cycle, or sets of chorded cycles, or cycles with multiple chords
exist, or chorded cycles with additional properties, see
the survey \cite{Gou} and references therein. 
Gould   asked in \cite{Gou} the following question:

\begin{question} \label{qu1}
What spectral conditions imply a graph contains a chorded cycle?
\end{question}

Given a graph $G$, denote by $\rho(G)$ the spectral radius of $G$ and $q(G)$ the signless Laplacian spectral radius of $G$. Let $K_{a,b}$ be the complete bipartite graph with partite sizes $a$ and $b$.

Zheng et al. \cite{ZHW} answered Gould's question by showing that any graph $G$ of order $n\ge 6$ with $\rho(G)\ge \rho(K_{2, n-2})$ contains a chorded cycle unless $G\cong K_{2,n-2}$.
We give different answers via the signless Laplacian spectral radius.

For positive integers $n$ and $c$ with $n\ge 3$ and $0\le c\le \lfloor \frac{n-1}{2}\rfloor$, let
$F_{n,c}$ be the graph obtained from $K_{1, n-1}$ by adding $c$ edges that are
pairwise nonadjacent. In particular, $F_{n,0}=K_{1, n-1}$.
Let $F_n=F_{n, \lfloor \frac{n-1}{2}\rfloor}$.  For odd $n$, $F_n$ is known as the friendship graph (or windmill graph) due to the Friendship Theorem \cite{ERS}:
a graph in which any two distinct vertices have exactly one common neighbor  has a vertex joined to all others. The main results are as follows.

\begin{theorem} \label{You}
If $G$ is a graph of order $n\ge 4$ with $q(G)\ge n$, then $G$ contains a chorded cycle unless $G\cong K_{2,n-2}$ or $G\cong F_{n,c}$ for some $c=0,\dots, \left\lfloor \frac{n-1}{2}\right\rfloor$.
\end{theorem}

\begin{theorem} \label{You2}
If $G$ is a graph of order $n\ge 4$ with
\[
q(G)\ge \begin{cases}
\frac{1}{2}(n+2+\sqrt{(n-2)^2+8}) & \mbox{if $n$ is odd}, \\[2mm]
\frac{1}{2}(n+1+\sqrt{(n-1)^2+8}) & \mbox{if $n$ is even},
\end{cases}
\]
then $G$ contains a chorded cycle unless $G\cong F_n$.
\end{theorem}

Observe that $K_{r,n-r}$ with $3\le r\le \lfloor \frac{n}{2}\rfloor$ satisfies the condition in Theorem \ref{You} but does not satisfies the condition in Theorem \ref{You2},  as $q(K_{r,n-r})=n$.  Thus,
Theorem \ref{You} is better than Theorem \ref{You2}, but
both theorems give  answers to Question \ref{qu1}. Actually, we will derive Theorem \ref{You2} from Theorem \ref{You}. Evidently, the conditions in both theorems are sharp.

Theorem \ref{You2} may be restated as: If $G$ is a graph of order $n\ge 4$ with no chorded cycles, then
\[
q(G)\le \begin{cases}
\frac{1}{2}(n+2+\sqrt{(n-2)^2+8}) & \mbox{if $n$ is odd}, \\[2mm]
\frac{1}{2}(n+1+\sqrt{(n-1)^2+8}) & \mbox{if $n$ is even},
\end{cases}
\]
with equality if and only if $G\cong F_n$. In this sense, it belongs to the problem of determining the maximum spectral radius or signless Laplacian spectral radius of $\mathcal{F}$-free graphs, where $\mathcal{F}$ is a class of given graphs. Here $\mathcal{F}$ is the class of chorded cycles with no restrictions on order and size. Related work, for example, when $\mathcal{F}$ consists of a single short cycle or several short cycles,  received much attention, see, e.g. \cite{CWZ,TCC,ZWF}.

If a graph $G$ contains no chorded cycles, then $G$ is $K_{3,3}$-free. However, if $G$ is $K_{3,3}$-free, it can contain a chorded cycle. For example, $K_{2,n-2}$ for $n\ge 4$  with an edge added and the join of $P_1$ and $P_{n-1}$ with $n\ge 4$ are such graphs. Similar cases occurs
if  a graph does not contain a cycle of length $4$,  a cycle of length $6$, or a theta-graph consisting of three  internally disjoint paths of length two with
the same two distinct end vertices.

\section{Preliminaries}

Let $G$ be a graph with vertex set $V(G)$ and edge set $E(G)$.
For $v\in V(G)$, we denote by $N_G(v)$ the neighborhood of $v$ in $G$, and  $d_G(v)$  the degree of $v$ in $G$.

For an $n$-vertex graph $G$, the adjacency matrix of $G$ is the $n\times n$ matrix
$A(G)=(a_{uv})_{u,v\in V(G)}$, where $a_{uv}=1$ if $uv\in E(G)$ and $a_{uv}=0$ otherwise,
and the signless Laplacian matrix of  $G$  is the $n\times n$ matrix
$Q(G)=(q_{uv})_{u,v\in V(G)}$, where
\[
q_{uv}=\begin{cases}
d_G(u) & \mbox{if $u=v$,}\\
1     & \mbox{if $u\ne v$ and $uv\in E(G)$,}\\
0     & \mbox{if $u\ne v$ and $uv\not\in E(G)$}.
\end{cases}
\]
That is, $Q(G)$ is equal to the sum of the  diagonal
degree matrix and the adjacency matrix of $G$.
The spectral radius of $G$ is the largest eigenvalue of $A(G)$, and
the signless Laplacian spectral radius of $G$ is the largest eigenvalue of $Q(G)$.

For a graph $G$ with $\emptyset\ne S\subseteq V(G)$, denote by $G[S]$ the subgraph of $G$ induced by $S$.
For a graph $G$ with two nonadjacent vertices $u$ and $v$, denote by $G+uv$ the graph obtained from $G$ by adding the edge $uv$.

The following lemma is an immediate consequence of the Perron-Frobenius Theorem.

\begin{lemma}\label{addedges}
Let $G$ be a  graph and $u$ and $v$  two nonadjacent vertices of $G$. If $G+uv$ is connected,  then $q(G+uv)> q(G)$.
\end{lemma}

We need two upper bounds for the signless Laplacian spectral radius.
A graph  is said to be semi-regular if it is bipartite  and  vertices in the same partite set have the same degree.

\begin{lemma}\label{qmu}\cite{FG} For a nontrivial connected graph $G$,
$q(G)\le \max\{d_G(u)+\frac{1}{d_G(u)}\sum_{v\in N_G(u)} d_G(v):u\in V(G) \}$ with equality  if and only if $G$ is regular or  semi-regular.
\end{lemma}

\begin{lemma}\label{qedge}\cite{CRS} For a connected graph $G$,
$q(G)\le \max\{d_G(u)+d_G(v): uv\in E(G)\}$ with  equality if and only if $G$ is regular or  semi-regular.
\end{lemma}

Recall that the Laplacian matrix $L(G)$ of a graph $G$ is defined as  the difference between the  diagonal
degree matrix and the adjacency matrix of $G$.

\begin{lemma} \cite{Me} \label{bb1}  For a graph $G$, $Q(G)$ and $L(G)$ have the same spectrum if and only if $G$ is a bipartite graph.
\end{lemma}

For a graph $G$, let $\mu(G)$ be the largest Laplacian eigenvalue of $G$, i.e, the largest eigenvalue of $L(G)$. Denote by $\overline{G}$ the complement of $G$.

\begin{lemma} \cite{GZ,Me} \label{bb2}
If $G$ is a graph of order $n\ge 2$, then $\mu(G)\le n$  with equality if and only if $\overline{G}$ is disconnected.
\end{lemma}
If $G$ is a connected graph, then $Q(G)$ is irreducible and so by Perron-Frobenius Theorem, there exists a unique unit positive eigenvector of $Q(G)$ corresponding to $q(G)$, which we call the $Q(G)$-Perron vector. For $Q(G)$-Perron vector $\mathbf{x}$, $x_u$ denotes the entry of  $\mathbf{x}$ at $u\in V(G)$.

\begin{lemma}\label{perron}\cite{CRS,HZ}
Let $G$ be a connected graph with $u,v\in V(G)$. Suppose that $\emptyset\ne N\subseteq N_G(v)\setminus (N_G(u)\cup \{u\})$. Let $G'$ be the graph
obtained from $G$ by deleting the edges $vw$ and adding the edges $uw$ for $w\in N$.
If the $Q(G)$-Perron vector $\mathbf{x}$ satisfies that $x_u\ge x_v$, then $q(G')>q(G)$.
\end{lemma}

Let $G\cup H$ be the disjoint union of graphs $G$ and $H$. The disjoint union of $k$ copies of a graph $G$ is denoted by $kG$.
The join of disjoint graphs $G$ and $H$, denoted by $G\vee H$, is the graph obtained from $G\cup H$ by adding all possible edges between vertices in $G$ and vertices in $H$.
Denote by $K_n$ the complete graph of order $n$. Then
\[
F_{n,c}=K_1\vee ((n-1-2c) K_1\cup c K_2).
\]

From Lemma \ref{addedges} and the fact that $q(K_{1, n-1})=n$, we have the following result.

\begin{lemma}\label{fn}
For  $1\le c\le \lfloor\frac{n-1}{2}\rfloor$, $n<q(F_{n,c})\le q(F_n)$ with right equality if and only if $F_{n,c}\cong F_n$.
\end{lemma}

Let $G$ be a graph with vertex partition $V_1\cup \dots\cup V_m$. For $1\le i<j\le m$, set $Q_{ij}$ to be the submatrix of $Q(G)$ with rows corresponding to vertices in $V_i$ and columns corresponding to vertices in $V_j$. The quotient matrix of $Q(G)$ with respect to the partition $V_1\cup \dots V_m$ is denoted by $B=(b_{ij})$, where $b_{ij}=\frac{1}{|V_i|}\sum_{u\in V_i}\sum_{v\in V_j}q_{uv}$. If $Q_{ij}$ has constant row sum, then we say $B$ is an equitable quotient matrix.
The following lemma is an immediate consequence of \cite[Lemma 2.3.1]{BH}

\begin{lemma}\label{quo}
For a connected graph $G$, if $B$ is an equitable quotient matrix of $Q(G)$, then $q(G)$ is equal to the largest eigenvalue of $B$.
\end{lemma}

For positive integer $n\ge 5$ and nonnegative integers $a_1$, $a_2$, $b_1$ and $b_2$ with $n=3+a_1+2a_2+b_1+2b_2$,
let $H^{(n)}_{a_1,a_2,b_1,b_2}$ be the graph obtained from a triangle $uvwu$ by attaching $a_1$ edges and $a_2$ triangles at $u$ and attaching $b_1$ edges and $b_2$ triangles at $v$, see Fig. \ref{H}.
Let
\[
\mathcal{H}^{(n)}=\{H^{(n)}_{a_1,a_2,b_1,b_2}: a_1,a_2,b_1, b_2\ge 0, a_1+a_2\ge 1,b_1+b_2\ge  1,n=3+a_1+2a_2+b_1+2b_2\}.
\]

\begin{figure}[htbp]
\centering
\includegraphics[width =8cm]{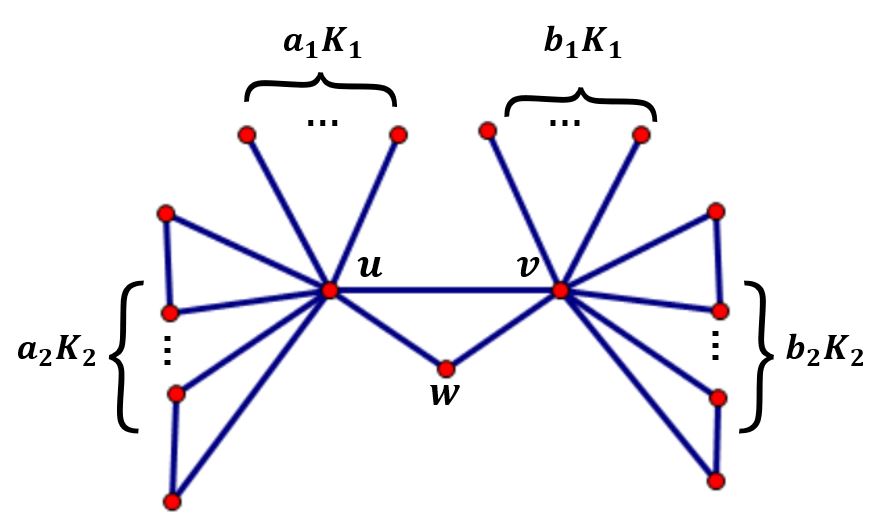}
\caption{The graph $H^{(n)}_{a_1,a_2,b_1,b_2}$.}
\label{H}
\end{figure}

\begin{lemma} \label{MM}  For even $n\ge 6$ and $a=\frac{n-4}{2}$, $q(H^{(n)}_{0,a,1,0})<n$.
\end{lemma}
\begin{proof}
Let $G=H^{(n)}_{0,a,1,0}$.  Let $V_1$ be the set of neighbors of $u$ with degree two different from $w$ in $G$ and
let $v'$ be the unique vertex of degree one in $G$.
With respect to  partition $V(G)=\{u\}\cup V_1\cup \{w\}\cup\{v\}\cup\{v'\}$, $Q(G)$ has an equitable quotient matrix $B$ with \[
B=\begin{pmatrix}
n-2&n-4&1&1&0\\
1&3&0&0&0\\
1&0&2&1&0\\
1&0&1&3&1\\
0&0&0&1&1
\end{pmatrix}.
\]
By Lemma \ref{quo}, $q(G)$ is the largest eigenvalue of $B$, which is equal to the largest root of $f(\lambda)=0$, where
\[
f(\lambda)=\lambda^5-(n+7)\lambda^4+(8n+11)\lambda^3-(21n-11)\lambda^2+(21n-32)\lambda+12-6n.
\]
It may be checked that
\begin{align*}
f'(\lambda)&=5\lambda^4-4(n+7)\lambda^3+3(8n+11)\lambda^2-2(21n-11)\lambda+21n-32\\
&\ge f'(n)=(n^2-4n-9)n^2+43n-32 >0
\end{align*}
if $\lambda\ge n$.
As $f(n)=(n-2)(n^3-8n^2+16n-6)>0$, the largest root of $f(\lambda)$ is less than $n$. So $q(G)<n$, as desired.
\end{proof}

\begin{lemma}\label{ab}
For any $G\in \mathcal{H}^{(n)}$ with $n\ge 5$, $q(G)<n$.
\end{lemma}
\begin{proof}
If $n=5$, then the result follows by a direct calculation and the fact that $\mathcal{H}^{(5)}=\{H^{(5)}_{1,0,1,0}\}$.

Suppose  that $n\ge 6$.
Let $G$ be a graph having the largest signless Laplacian spectral radius in $\mathcal{H}^{(n)}$, say $G=H^{(n)}_{a_1,a_2,b_1,b_2}$. By Lemma \ref{addedges}, we have $a_1,b_1=0,1$.
Let $\mathbf{x}$ be the $Q(G)$-Perron vector. Assume that $x_u\ge x_v$.
If $b_2=0$, then as $b_1+b_2\ge 1$ and $b_1\le 1$, we have $b_1=1$.
If $b_2\ge 1$, then we have by Lemma \ref{perron} that $b_1=0$ and $b_2=1$.

For $z\in V(G)$, let \[
\Delta_G(z)=d_G(z)+\frac{1}{d_G(z)}\sum_{z'\in N_G(z)}d_G(z').
\]
Then
\begin{align*}
\Delta_G(w)&=d_G(w)+\frac{1}{d_G(w)}\sum_{z\in N_G(w)}d_G(z)\\
&=2+\frac{1}{2}\left(a_1+2a_2+2+b_1+2b_2+2\right)\\
&=2+\frac{n+1}{2}\\
&<n.
\end{align*}
It is easy to verified that each vertex $z\in V(G)\setminus\{u,v,w\}$, $\Delta_G(z)<n$.

\noindent
{\bf Case 1.} $b_1=1$ and $b_2=0$.

If $a_1=0$, then $G=H^{(n)}_{0,a_2,1,0}$, so we have by Lemma \ref{MM} that $q(G)<n$.

Suppose that $a_1=1$. Then $a_2=\frac{n-5}{2}$, so
\begin{align*}
\Delta_G(u)&=d_G(u)+\frac{1}{d_G(u)}\sum_{z\in N_G(u)}d_G(z)\\
&=n-2+\frac{1}{n-2}\left(1+2(n-5)+2+3\right)\\
&=n
\end{align*}
and
\[
\Delta_G(v)=d_G(v)+\frac{1}{d_G(v)}\sum_{z\in N_G(v)}d_G(z)=3+\frac{1}{3}\left(n-2+2+1 \right)<n.
\]
As $G$ is neither regular nor semi-regular, we have by Lemma \ref{qmu} that $q(G)<\max\{\Delta_G(z):z\in V(G) \}=n$, as desired.

\noindent
{\bf Case 2.} $b_1=0$ and $b_2=1$.

If $n=6$, then $a_1=1$, $a_2=0$, and $G=H_{1,0,0,1}\cong H_{0,1,1,0}$, so we have by Lemma \ref{MM} that $q(G)<6=n$.

Suppose that $n\ge 7$. As $n=3+a_1+2a_2+2$, we have $a_1+2a_2=n-5$, so
\begin{align*}
\Delta_G(u)&=d_G(u)+\frac{1}{d_G(u)}\sum_{z\in N_G(u)}d_G(z)\\
&=a_1+2a_2+2+\frac{1}{a_1+2a_2+2}\left(a_1+4a_2+2+4\right)\\
&=n-3+\frac{n-1+2a_2}{n-3}\\
&< n
\end{align*}
and
\begin{align*}
\Delta_G(v)=&d_G(v)+\frac{1}{d_G(v)}\sum_{z\in N_G(v)}d_G(z)\\
&=4+\frac{1}{4}\left(a_1+2a_2+2+2+2\cdot 2 \right)\\
&= 6+\frac{n-5}{4}\\
&<n.
\end{align*}
Now, by Lemma \ref{qmu}, $q(G)\le \max\{\Delta_G(z):z\in V(G) \}<n$.
\end{proof}

\section{Proof of Theorems \ref{You} and \ref{You2}}

Now we are ready to prove Theorem \ref{You}.
\begin{proof}[Proof of Theorem \ref{You}]
Suppose by contradiction that $G$ contains no chorded cycles. By Lemma \ref{addedges}, we may assume that $G$ is connected.
Note that neither $K_{1,n-1}$ nor $K_{2,n-2}$ contains a chorded cycle, and $q(K_{1,n-1})=q(K_{2,n-2})=n$. Moreover,
for any $c=1,\dots, \left\lfloor \frac{n-1}{2}\right\rfloor$,
$F_{n,c}$  contains no chorded cycles, and  $q(F_{n,c})>n$ by Lemma \ref{fn}.
Thus, it suffices to show that $G\cong K_{1,n-1},K_{2,n-2}$ or $G\cong  F_{n,c}$ for some $c=1,\dots, \left\lfloor \frac{n-1}{2}\right\rfloor$.

Let $uv$ be an edge of $G$ such that $d_G(u)+d_G(v)=\max\{d_G(w)+d_G(z):wz\in E(G)\}$.
By Lemma \ref{qedge} and the assumption, we have
\[
d_G(u)+d_G(v)\ge q(G)\ge n.
\]
We consider  cases $d_G(u)+d_G(v)=n$ and $d_G(u)+d_G(v)\ge n+1$ separately.

\noindent
{\bf Case 1.} $d_G(u)+d_G(v)=n$.

In this case, $d_G(u)+d_G(v)=q(G)=n$.
By Lemma \ref{qedge} again, $G$ is regular or semi-regular.

Suppose first that $G$ is regular. Then $G$ is regular of degree $\frac{n}{2}$, so $n$ is even.  If $n\ge 6$, then
the aforementioned Czipser's result implies that $G$ contains a chorded cycle, a contradiction. So $n=4$ and $G\cong K_{2,2}$, as desired.

Suppose next that $G$ is semi-regular.
Let $(V_1,V_2)$ be the bipartition of $G$.
We claim that $G$ is complete bipartite. Otherwise, there exists an edge $uv$ in $\overline{G}$ with $u\in V_1$ and $v\in V_2$,  so both $G$ and
$\overline{G}$ are connected. By Lemmas \ref{bb1} and \ref{bb2},  $q(G)=\mu(G)<n$, a contradiction.
It then follows that $G\cong K_{a,n-a}$ for some $a=1,\dots,\lfloor \frac{n}{2}\rfloor$. If $a\ge 3$, then we have  by an easy check or by Czipser's result that $G$ contains a chorded cycle, a contradiction. So $G\cong K_{1,n-1},K_{2,n-2}$, as desired.

\noindent
{\bf Case 2.} $d_G(u)+d_G(v)\ge n+1$.

We claim that $|N_G(u)\cap N_G(v)|\le 1$ as otherwise $G$ contains a chorded cycle.
Then \[
|N_G(u)\cup N_G(v)|= d_G(u)+d_G(v)-|N_G(u)\cap N_G(v)|\ge n+1-1=n,
\]
so $V(G)=N_G(u)\cup N_G(v)$ and $|N_G(u)\cap N_G(v)|=1$, say $w\in N_G(u)\cap N_G(v)$.

Let $N_1=N_G(u)\setminus \{v,w\}$ and $N_2=N_G(v)\setminus\{u,w\}$. Suppose that $N_1\ne\emptyset$ and $N_2\ne\emptyset$.
If there is an edge $z_1z_2$ with $z_i\in N_i$ for $i=1,2$, then $z_1uwvz_2z_1$ is a cycle with a chord $uv$, which is a chorded cycle, a contradiction. So there is no edge between vertices in $N_1$ and vertices in $N_2$.
Then $G\in \mathcal{H}^{(n)}$ and so by Lemma \ref{ab}, $q(G)<n$, a contradiction.
So $N_1=\emptyset$ or $N_2=\emptyset$.

Assume that $N_2=\emptyset$. Then $d_G(u)=n-1$.
As $G$ has no chorded cycles, $G[N_G(u)]$ contains no copies of $K_{1,2}$, so $G[N_G(u)]\cong (n-1-2c) K_1\cup c K_2$ for some positive integer $c\le \lfloor\frac{n-1}{2}\rfloor$. So $G\cong  F_{n,c}$.
\end{proof}

From Theorem \ref{You}, we can prove Theorem \ref{You2}.

\begin{proof}[Proof of Theorem \ref{You2}]
If $q(G)\ge q(F_n)$, then
by the fact that $q(K_{1,n-1})=q(K_{2,n-2})=n$, Lemma \ref{fn} and Theorem \ref{You},
$G$ contains a chorded cycle unless $G\cong F_n$.

Applying
Lemma \ref{quo} we find that
$q(F_n)$ is equal to the largest eigenvalue of the matrix
\[
\begin{pmatrix}
n-1 & n-1\\
1 & 3
\end{pmatrix}
\]
if $n$ is odd
and
the largest eigenvalue of the matrix
\[
\begin{pmatrix}
n-1 & n-2 & 1\\
1 & 3 & 0\\
1 & 0 & 1
\end{pmatrix}
\]
if $n$ is even. In the latter case, $q(F_n)$ is equal to the largest root of $\lambda^3-(n+3)\lambda^2+3n\lambda-(2n-4)=0$, that is,
$(\lambda-2)(\lambda^2-(n+1)\lambda+n-2)=0$.  So
\[
q(F_n)=\begin{cases}
\frac{1}{2}(n+2+\sqrt{(n-2)^2+8}) & \mbox{if $n$ is odd}, \\[2mm]
\frac{1}{2}(n+1+\sqrt{(n-1)^2+8}) & \mbox{if $n$ is even}.
\end{cases}
\]
Now the result follows.
\end{proof}

By Theorem \ref{You} and Lemma \ref{addedges}, we have the following corollary.

\begin{corollary}
If $G$ is a graph of order $n\ge 4$ containing no chorded cycles and $G\ncong K_{2,n-2}$, $F_{n,c}$ for any $c=0,\dots,\lfloor\frac{n-1}{2}\rfloor$, then
\[
q(G)<n=q(K_{2,n-2})=q(F_{n,0})<\dots <q(F_{n,\lfloor (n-1)/2\rfloor}).
\]
\end{corollary}

A theta-graph consists of at least three  internally disjoint paths with
the same two distinct end vertices. A chorded cycle is a theta graph, but a theta graph is not necessarily a chorded cycle. By Theorems \ref{You} and \ref{You2}
and the fact that $K_{2,k}$ with $k\ge 3$ is a theta-graph, we have

\begin{theorem}
Suppose that $G$ is a graph of order $n\ge 4$.

(i) If $q(G)\ge n$, then $G$ contains a theta-graph unless  $G\cong F_{n,c}$ for some $c=0,\dots, \left\lfloor \frac{n-1}{2}\right\rfloor$ if $n\ge 5$, and $G\cong C_4$ or $G\cong F_{n,c}$ for some $c=0,1$ if $n=4$.

(ii) If
\[
q(G)\ge \begin{cases}
\frac{1}{2}(n+2+\sqrt{(n-2)^2+8}) & \mbox{if $n$ is odd}, \\[2mm]
\frac{1}{2}(n+1+\sqrt{(n-1)^2+8}) & \mbox{if $n$ is even},
\end{cases}
\]
then $G$ contains a theta-graph unless $G\cong F_n$.
\end{theorem}

\section{Concluding remarks}

For positive integers $n$ and $k$, let $g_k(n)$ denote the maximum size of a graph of order $n$ that does not contain a cycle with $k$ chords incident to a vertex on the cycle.
Alon \cite{Alo} used $g_k(n)$ to bound an anti-Ramsey function from above.
Erd\H{o}s conjectured that
$g_k(n)=(k+1)n-(k+1)^2$ if $n\ge 2k+2$. Lewin (see \cite[p.~398, no.~12]{Bol}) showed that Erd\H{o}s' conjecture is not true if $2k\le n\le \frac{5(k-1)}{2}$. Bollob\'{a}s \cite[p.~398, no.~12]{Bol} conjectured that there is a function $n(k)$ such that
$g_k(n)=(k+1)n-(k+1)^2$ for $n\ge n(k)$.  Jiang \cite{Ji}  confirmed Bollob\'{a}s' conjecture by showing
$n(k)\le 3k+3$. Thus, if $G$ is a graph of order $n$ with $n\ge 3k+3$  that does not contain a cycle with $k$ chords incident to a vertex on the cycle, then $|E(G)|\le (k+1)n-(k+1)^2$, and it is attained if $G\cong K_{k+1, n-k-1}$.
By Theorem \ref{You2}, the extremal graph for signless Laplacian spectral radius behaves quite different from that for the size among graphs without a chorded cycle. For $k\ge 2$,
similar cases occur for graphs without a cycle with $k$ chords incident to a vertex on the cycle, as
$q(F_n)>n=q(K_{k+1, n-k-1})$ and both $F_n$ and $K_{k+1, n-k-1}$ do not contain a cycle with $k$ chords incident to a vertex on the cycle. However, $K_{2,n-2}$ maximizes both the spectral radius and the size
among $n$-vertex graphs without a chorded cycle. On one hand, the spectral sufficient conditions that imply a graph contains a chorded cycle via spectral radius and signless Laplacian spectral radius may be extended to $\alpha$-spectral radius for $0\le \alpha<\frac{1}{2}$ and $\frac{1}{2}\le \alpha<1$, respectively \cite{GuoZh}.
On the other hand, we conjecture that for some $n_0(k)$ with $k\ge 2$, if $G$ is a graph of order $n\ge n_0(k)$ with $\rho(G)\ge \rho(K_{k+1, n-k-1})$, then $G$ contains a cycle with $k$ chords incident to a vertex on the cycle  unless $G\cong K_{k+1, n-k-1}$.

\bigskip

\noindent
{\bf Declaration of competing interest}

There is no competing interest.

\bigskip

\noindent
{\bf Data availability}

No data was used for the research described in the article.

\bigskip
\noindent {\bf Acknowledgement}

The authors thank the referees for their constructive comments and suggestions.
This work was supported by the National Natural Science Foundation of China (No.~12071158).

\end{document}